\documentclass[12pt, reqno]{amsart}
\usepackage{amsmath, amsthm, amscd, amsfonts, amssymb, graphicx, color, mathrsfs}
\usepackage[bookmarksnumbered, colorlinks, plainpages]{hyperref}

\usepackage{slashed}
\usepackage[capitalize]{cleveref}

\newtheorem{theorem}{Theorem}[section]
\newtheorem{lemma}[theorem]{Lemma}

\newtheorem{proposition}[theorem]{Proposition}
\newtheorem{corollary}[theorem]{Corollary}

\theoremstyle{definition}
\newtheorem{definition}[theorem]{Definition}

\theoremstyle{remark}
\newtheorem{remark}[theorem]{Remark}
\numberwithin{equation}{section}

\newcommand*\diff{\mathop{}\!\mathrm{d}}
\DeclareMathOperator{\supp}{supp}

\crefname{equation}{}{}

\begin{document}
\setcounter{page}{1}

\title[Estimates for operators on the torus. III]{Estimates for pseudo-differential operators on the torus revisited. III}

\author[D. Cardona]{Duv\'an Cardona}
\address{
  Duv\'an Cardona:
  \endgraf
  Department of Mathematics: Analysis, Logic and Discrete Mathematics
  \endgraf
  Ghent University, Belgium
  \endgraf
  {\it E-mail address} {\rm duvanc306@gmail.com, duvan.cardonasanchez@ugent.be}
  }

\author[M. A. Mart\'inez]{Manuel Alejandro Mart\'inez}
\address{
  Manuel Alejandro Mart\'inez
  \endgraf
  Department of Mathematics
  \endgraf
 Universidad del Valle de Guatemala, Guatemala
  \endgraf
  {\it E-mail address} {\rm mar21403@uvg.edu.gt, manuelalejandromartinezf@gmail.com}
  }

\thanks{ Manuel Alejandro Mart\'inez has been supported by the {\it{Liderazgo en Ciencias}} scholarship of the Universidad del Valle de Guatemala.  Duv\'an Cardona is supported  by the FWO  Odysseus  1  grant  G.0H94.18N:  Analysis  and  Partial Differential Equations, by the Methusalem programme of the Ghent University Special Research Fund (BOF)
(Grant number 01M01021) and has been supported by the FWO Fellowship
Grant No 1204824N of the Belgian Research Foundation FWO}

     \keywords{Discrete Fourier analysis - Oscillating singular integrals - Periodic pseudo-differential operators - Torus}
    \subjclass[2010]{Primary 22E30; Secondary 58J40.}

\begin{abstract}
    This paper finishes  the goal of the authors started in two previous manuscripts dedicated to revisiting the continuity properties of toroidal pseudo-differential operators with symbols in the H\"ormander classes. Here we prove pointwise estimates in terms of the Fefferman-Stein sharp maximal function and of the Hardy-Littlewood maximal function. Combining these estimates with the properties of  Muckenhoupt's weight class $A_p$ we obtain boundedness theorems for pseudo-differential operators between weighted Lebesgue spaces on the torus $L^p(w)$. These results are given in the context of the global symbolic analysis defined on $\mathbb{T}^n\times \mathbb{Z}^n$ as developed by Ruzhansky and Turunen by using discrete Fourier analysis, and extend those of Park and Tomita available in the Euclidean case.   Moreover, we include continuity results on Sobolev spaces $W^s_p$ and on Besov spaces $B^s_{p,q}$ on the torus. Our techniques are taken from Park and Tomita \cite{park-tomita} and we consider its toroidal extension here for the completeness of the boundedness of toroidal pseudo-differential operators with respect to the current literature.
\end{abstract} 

\maketitle
\tableofcontents
\allowdisplaybreaks

\section{Introduction}

By continuing our program of studying the continuity properties of pseudo-differential operators on the torus, first in the context of Lebesgue spaces $L^p$, as discussed in \cite{Cardona:Martinez}, and then in the context of Hardy spaces $H^p$, as discussed in \cite{Cardona:Martinez-2}, this paper focuses on pointwise estimates  of toroidal pseudo-differential operators in terms of the Fefferman-Stein sharp maximal function and of the Hardy-Littlewood maximal function. Moreover, we study the continuity properties of such operators in weighted Lebesgue spaces $L^p(w)$, Sobolev spaces $W^s_p$, and Besov spaces $B^s_{p,q}$.

Fefferman and Stein introduced the sharp maximal function $\mathcal{M}^\#$ in \cite{fefferman-stein}, which serves as a way to characterize the norm of the space of functions with bounded mean oscillation $\mathrm{BMO}(\mathbb{R}^n)$. Moreover, they proved that it satisfies an upper bound with respect to the $L^p$-norm of integrable functions, namely, it satisfies that $\|f\|_{L^p} \lesssim \|\mathcal{M}^\#f\|_{L^p}$. Here, and in the remainder of this paper, $A\lesssim B$ means that there exists a constant $C>0$, such that $A\leq CB$. Moreover, it has been proved that if $T$ is a Calder\'on-Zygmund operator, then we have the pointwise inequality $\mathcal{M}^\#(Tf)(x) \lesssim \mathrm{M}_rf(x)$, where $\mathrm{M}_r$ is the $L^r$ version of the Hardy-Littlewood maximal function. Combining these two estimates we obtain continuity estimates of $T$ from $L^p$ into itself. Namely, that
\begin{equation}
    \|Tf\|_{L^p} \lesssim\|\mathcal{M}^\#(Tf)\|_{L^p} \lesssim\| \mathrm{M}_rf\|_{L^p} \lesssim \|f\|_{L^p}.
    \label{eq:technique}
\end{equation}
This technique has been widely employed in a variety of works in harmonic analysis, see \cite{chanillo-torchinsky, fefferman-stein}. Moreover, Muckenhoupt proved that weights $w$ in the class $A_p$, satisfy the following estimate $\|\mathrm{M}f\|_{L^p(w)} \lesssim \| f\|_{L^p(w)}$, for $1<p<\infty$, see \cite{muckenhoupt}. Combining these two facts, Chanillo and Torchinsky \cite{chanillo-torchinsky} proved the following continuity result for weighted Lebesgue spaces $L^p(w)$ for Euclidean pseudo-differential operators with symbols in the H\"ormander classes $S^{m}_{\rho,\delta} (\mathbb{R}^n\times \mathbb{R}^n)$.

\begin{theorem} \cite{chanillo-torchinsky}
    Let $0\leq \delta < \rho<1$, and let $2\leq p < \infty$. Suppose $w \in A_{p/2} $ and $\sigma \in S^{-n(1-\rho)/2}_{\rho,\delta} (\mathbb{R}^n\times \mathbb{R}^n)$. Then $T_\sigma$ extends to a bounded operator from $L^p(w)$ into itself. 
\end{theorem}

The estimate above has been extended to the case $0<\delta=\rho<1$, but with a certain restriction on the range of $\rho$, by Miyachi and Yabuta.

\begin{theorem} \cite{miyachi-yabuta}
    Let $1<r\leq 2$, let $0<\rho\leq r/2$, and let $\rho<1$. Suppose that $m\leq -n(1-\rho)/r $, and let $\sigma \in S^m_{\rho,\rho} (\mathbb{R}^n\times \mathbb{R}^n)$. Then, it holds that
    \begin{equation*}
        \mathcal{M}^\# (T_\sigma f)(x) \lesssim\mathrm{M}_rf(x), \,x\in \mathbb{R}^n,
    \end{equation*}
    for all $f \in \mathcal{S}(\mathbb{R}^n)$.
\end{theorem}
For more results related to the previous one, we refer the reader to \'Alvarez and Hounie \cite{alvarez-hounie}. The previous theorem was extended by Park and Tomita \cite{park-tomita} to the range $0<\rho<1$. For that purpose, they defined the $L^r$ version of the inhomogeneous sharp maximal function. 
\begin{equation*}
    \mathfrak{M}^\#_rf(x) := \sup_{Q:\ell(Q) \leq 1} \inf_{c_Q \in \mathbb{C}}\left(
        \frac{1}{|Q|} \int_Q |f(y) - c_Q|^r\diff y
        \right)^{1/r} + \sup_{Q:\ell(Q) > 1}\left(
        \frac{1}{|Q|} \int_Q |f(y)|^r\diff y
        \right)^{1/r},
\end{equation*}
where the supremum ranges over cubes with sides parallel to the axis. However, this coincides with the homogeneous case in the torus $\mathbb{T}^n$. Now, we present the toroidal extension of Park and Tomita's result \cite[Theorem~1.1]{park-tomita}. 

\begin{theorem}
    Let $1<r\leq 2$ and let $0<\rho<1$. Suppose that $m\in \mathbb{R}$ is such that $m\leq-n(1-\rho)/r$, and let us consider the periodic symbol $\sigma \in S^m_{\rho,\rho}(\mathbb{T}^n \times \mathbb{Z}^n)  $. Then it holds that 
    \begin{equation*}
        \mathcal{M}^\#_r(T_\sigma f)(x) \lesssim\mathrm{M}_r f(x),\,x\in \mathbb{T}^n,
    \end{equation*}
    for all $f \in C^\infty(\mathbb{T}^n)$.
    \label{theo:main1}
\end{theorem}

Hence we can use an argument similar as in \cref{eq:technique} to prove the $L^p(w)$-boundedness for toroidal pseudo-differential operators, where $w$ is in the Muckenhoupt's weight class $A_{p/r}$. As Park and Tomita \cite{park-tomita}, first we apply a Hardy-Littlewood decomposition on the frequency domain of the symbol $\sigma := \sigma(x, \xi)$ in order to obtain a sequence of operators $T_{\sigma_k}$. Then, we obtain uniform estimates in terms of the Schwartz kernels of the corresponding sequence of operators. We also take advantage of the $L^p$-$L^q$-boundedness result proved in \cite{Cardona:Martinez}. Finally, we prove \cref{theo:main1} and obtain $L^p(w)$ continuity as a corollary. 

On the other hand, Sobolev and Besov spaces have been of interest in the study of existence and uniqueness of solutions of systems of partial differential equations. Here, we prove versions of the $L^p$-$L^q$-boundedness result proved in \cite{Cardona:Martinez} in both, the Sobolev spaces $W^s_p(\mathbb{T}^n)$, and on Besov spaces $B^s_{p,q}(\mathbb{T}^n)$, considering even the case $\delta > \rho$. We present our main result in this topic as follows.

\begin{theorem}
    Let $0\leq\delta<1$, let $0<\rho\leq 1$, let $m\in\mathbb{R}$, and let $\sigma\in S^m_{\rho,\delta}(\mathbb{T}^n\times\mathbb{Z}^n)$. Then, $T_\sigma$ extends to a bounded operator from $W^{s}_p(\mathbb{T}^n)$ into $W^{s-\mu}_q(\mathbb{T}^n)$, and from $B^s_{p,r}(\mathbb{T}^n)$ into $B^{s-\mu}_{q,r}(\mathbb{T}^n)$, where $1<p\leq q<\infty,$ for any $s\in\mathbb{R},$ and $1\leq r\leq \infty,$ provided that the following conditions hold:
    \begin{enumerate}
        \item First,  $1<p\leq 2 \leq q$, and 
        \begin{equation*}
            \mu \geq m + n \left( \frac{1}{p} - \frac{1}{q} + \lambda
            \right).
        \end{equation*}
        \item Secondly, $2 \leq p \leq q$, and 
        \begin{equation*}
            \mu \geq m + n \left[ \frac{1}{p} - \frac{1}{q} + (1-\rho) \left( \frac{1}{2} - \frac{1}{p}
            \right)
            + \lambda
            \right].
        \end{equation*}
        \item Finally,  $p\leq q \leq 2$, and 
        \begin{equation*}
            \mu \geq m + n \left[ \frac{1}{p} - \frac{1}{q} + (1-\rho) \left( \frac{1}{q} - \frac{1}{2}
            \right)
            + \lambda
            \right].
        \end{equation*}
    \end{enumerate}
     In all the cases above we consider the efficiency parameter $\lambda := \max\{ 0, (\delta-\rho)/2 \}$. 
     \label{theo:main2}
\end{theorem}
Now, we proceed to discuss our main results.
\begin{remark}
    Notice that the case $s=0=\mu$ in \cref{theo:main2}, corresponds to the $L^p$-$L^q$ continuity result proved in \cite{Cardona:Martinez}. Moreover, let us recall that the $(\rho, \delta)$-H\"ormander symbol classes are not stable under the multiplication of test functions when $\rho < 1 - \delta$. Hence, these results cannot be obtained from its Euclidean counterpart considering the torus as a closed manifold and using on it a partition of unity.
\end{remark}
This work is organized as follows. In \cref{section:prelims} we present the calculus of pseudo-differential operators on the torus, as developed by Ruzhansky and Turunen in \cite{ruzhansky-turunen}. Moreover, we define the maximal operators of interest, and the functions spaces between which we will prove the continuity of toroidal pseudo-differential operators. In \cref{ssection:kernel} we present useful estimates in terms of the Schwartz kernel, and in \cref{ssection:boundedness} we prove the pointwise estimates in terms of maximal functions and the boundedness properties of toroidal pseudo-differential operators in weighted Lebesgue spaces $L^p(w)$, the Sobolev spaces $W^s_p$, and Besov spaces $B^s_{p,q}$.

\section{Preliminaries}
\label{section:prelims}
In this section we present the preliminaries about the Fourier analysis on the torus and on the toroidal pseudo-differential calculus as developed in \cite{ruzhansky-turunen}.

\begin{definition}[Periodic functions] 
    Let $X$ be a Banach space. A function $f: \mathbb{R}^n \rightarrow X$, is 1-periodic if $f(x) = f(x + k)$ for every $ x \in \mathbb{R}^n $ and all $ k \in \mathbb{Z}^n $. We can identify this function with one defined in $\mathbb{T}^n := \mathbb{R}^n / \mathbb{Z}^n $, which is a compact manifold without boundary. Moreover, the space of 1-perodic functions $m$ times continuously differentiable is denoted by $C^m(\mathbb{T}^n;X)$. The subset of these functions that have compact support is denoted by $C^m_0(\mathbb{T}^n;X)$.  The test functions are elements of the space $C^\infty(\mathbb{T}^n;X) := \bigcap_m C^m(\mathbb{T}^n;X)$. When $X=\mathbb{C}$, we will simply use $C^\infty(\mathbb{T}^n)$.
\end{definition}
In order to define the class of pseudo-differential operators on $\mathbb{T}^n$, we first need to define the group Fourier transform of 1-periodic smooth functions.
\begin{definition}[Group Fourier Transform on $\mathbb{T}^n$]
    For $f \in C^\infty(\mathbb{T}^n)$ we define the  \textit{toroidal Fourier transform} $\mathcal{F}_{\mathbb{T}^n}$ as follows
    \begin{equation*}
        (\mathcal{F}_{\mathbb{T}^n}f)(\xi) := \int_{\mathbb{T}^n} e^{-i2\pi x \cdot \xi}f(x)\diff x, 
    \end{equation*}
    as a function of $\xi \in \mathbb{Z}^n$.
\end{definition}
In view of the group structure of the torus $\mathbb{T}^n$, the corresponding frequency domain of 1-periodic smooth functions is the lattice $\mathbb{Z}^n$. Also, it can be proved that the corresponding functions $\mathcal{F}_{\mathbb{T}^n}f$ satisfy certain regularity conditions. In fact, they belong to the Schwartz space defined below.
Here we use $\langle\xi\rangle:= (1+|\xi|^2)^{1/2}$.
\begin{definition}[Schwartz space $\mathcal{S}(\mathbb{Z}^n) $]
    We say $\varphi \in \mathcal{S}(\mathbb{Z}^n)$, that is a \textit{rapidly decaying} function $\mathbb{Z}^n \rightarrow \mathbb{C}$, if for given $0<M<\infty$, there exists $C_{\varphi M} > 0$ such that 

    \begin{equation*}
        |\varphi(\xi)| \leq C_{\varphi M} \langle\xi\rangle^{-M} \quad , \quad \forall\, \xi \in \mathbb{Z}^n.
    \end{equation*}
\end{definition}
It can be proved that the toroidal Fourier transform is a bijection between $C^\infty(\mathbb{T}^n)$ and $\mathcal{S}(\mathbb{Z}^n)$ with inverse  $\mathcal{F}_{\mathbb{T}^n}^{-1}:  \mathcal{S}(\mathbb{Z}^n) \rightarrow C^\infty(\mathbb{T}^n)$ defined by
    \begin{equation*}
        (\mathcal{F}_{\mathbb{T}^n}^{-1}\varphi)(x) := \sum_{\xi \in \mathbb{Z}^n} e^{i2\pi x \cdot \xi} \varphi(\xi).
    \end{equation*}
In the Euclidean case, the Hörmander symbol classes are defined requiring some regularity conditions using partial derivatives in both, the space and frequency domain. However, in the case of the torus the frequency domain is the lattice $\mathbb{Z}^n$. Therefore, we define the natural analogous of the derivative in the discrete case: the partial difference operators.
\begin{definition}[Partial difference operators]
    Let $e_j \in \mathbb{Z}^n$ be such that $(e_j)_i$ is equal to 1, if $i=j$, and let it be equal to 0 otherwise. Then, for $\sigma(\xi):\mathbb{Z}^n \rightarrow \mathbb{C} $, we define 
    \begin{equation*}
        \Delta_{\xi_j}\sigma(\xi) := 
        \sigma(\xi + e_j) - \sigma(\xi) \quad \text{and} \quad \Delta^\alpha_\xi := \Delta_{\xi_1}^{\alpha_1} \cdots \Delta_{\xi_n}^{\alpha_n},
    \end{equation*}
    for any multi-index $\alpha \in \mathbb{N}^n_0$.
\end{definition}
Now we proceed to define the Hörmander classes of symbols on $ \mathbb{T}^n \times \mathbb{Z}^n $, as in Ruzhansky and Turunen \cite{ruzhansky-turunen}, that will be important in the quantization of pseudo-differential operators.
\begin{definition}[Hörmander classes on $ \mathbb{T}^n \times \mathbb{Z}^n $]
    Let $m \in \mathbb{R} $, and let $0 \leq \delta, \rho\leq 1$. We say that a function $\sigma:=\sigma(x, \xi)$ that is smooth on $x$ for any $\xi \in \mathbb{Z}^n$ belongs to the \textit{toroidal symbol class} $S^m_{\rho,\delta} (\mathbb{T}^n\times \mathbb{Z}^n) $ if 
    \begin{equation*}
        \left|\partial^\beta_x \Delta^\alpha_\xi \sigma(x, \xi)\right| \leq C_{\alpha\beta}\langle\xi\rangle^{m-\rho|\alpha| + \delta|\beta|},
    \end{equation*}
    for every $x \in \mathbb{T}^n$, $\alpha, \beta \in \mathbb{N}^n_0$ and $\xi \in \mathbb{R}^n$. Moreover, we say that $\sigma(x, \xi)$ has order $m$ and we define $S^{-\infty}_{\rho,\delta} (\mathbb{T}^n\times \mathbb{Z}^n) = \bigcap_{m\in\mathbb{R}}S^m_{\rho,\delta} (\mathbb{T}^n\times \mathbb{Z}^n) $.
\end{definition}

\begin{definition}[Pseudo-differential operators on $ \mathbb{T}^n \times \mathbb{Z}^n $]
    For a symbol $\sigma:=\sigma(x, \xi) \in S^m_{\rho,\delta} (\mathbb{T}^n\times \mathbb{Z}^n) $ we can associate a \textit{toroidal pseudo-differential operator} $T_\sigma$ from $C^\infty(\mathbb{T}^n)$ into itself, defined as  
    \begin{equation*}
        T_\sigma f(x):=\sum_{\xi \in \mathbb{Z}^n} e^{i2\pi x \cdot \xi}\sigma(x, \xi)(\mathcal{F}_{\mathbb{T}^n}f)(\xi),
    \end{equation*}
    which can be rewritten as 
    \begin{equation}
        T_\sigma f(x)=  \sum_{\xi \in \mathbb{Z}^n} \int_{\mathbb{T}^n} e^{i2\pi (x-y) \cdot \xi} \sigma(x, \xi)f(y) \diff y .
        \label{eq:pdo-def}
    \end{equation}
    The class of operators with symbols in $S^m_{\rho,\delta} (\mathbb{T}^n\times \mathbb{Z}^n) $ will be denoted  by $\Psi^m_{\rho,\delta} (\mathbb{T}^n\times \mathbb{Z}^n) $.
\end{definition}
    The toroidal Fourier transform and toroidal pseudo-differential operators can be extended by duality to the space of \textit{periodic distributions} $\mathcal{D}'(\mathbb{T}^n)$ consisting of continuous linear functionals on $C^\infty(\mathbb{T}^n)$. This extension allows us to define the Schwartz kernel of a toroidal pseudo-differential operator.
\begin{remark}[Schwartz kernel]
    In the sense of distributions described above, \cref{eq:pdo-def} can be rewritten as
    \begin{equation*}
         T_\sigma f(x)=\int_{\mathbb{T}^n}\left[\sum_{\xi \in \mathbb{Z}^n} e^{i2\pi(x - y) \cdot \xi} \sigma(x, \xi)\right]f(y)\diff y = \int_{\mathbb{T}^n}k(x, y)f(y)\diff y,
    \end{equation*}
    and we say that $k(x, y)$ is the \textit{Schwartz kernel} of the corresponding operator. Notice that when the order of $p(x, \xi)$ is less than $-n$, the series defining $k$ is absolutely convergent and this distribution agrees with a periodic function on $\mathbb{T}^n\times\mathbb{T}^n$.
\end{remark}
Since the proofs of the results in this paper only depend on the order of the toroidal pseudo-differential operators, then they extend to their adjoints. This is justified by the following result in \cite[Corollary~4.9.8]{ruzhansky-turunen}.
\begin{theorem}
    If $T \in \Psi^m_{\rho,\delta} (\mathbb{T}^n \times \mathbb{Z}^n) $, then its adjoint $T^* \in \Psi^m_{\rho,\delta} (\mathbb{T}^n \times \mathbb{Z}^n) $.
\end{theorem}
Now we introduce a particular and very useful toroidal pseudo-differential operator.
\begin{definition}[Bessel potential]
    We define the \textit{Bessel potential} $J^s$ as the pseudo-differential operator with symbol $\langle\xi\rangle^s$ for $s \in \mathbb{R}$.
\end{definition}

For our pourposes, we will use a smooth interpolation of a symbol on the phase space $\mathbb{T}^n \times \mathbb{Z}^n $ to obtain a symbol defined on the continuous phase space $\mathbb{T}^n \times \mathbb{R}^n $. This is possible in view of the results stated in \cite[section~4.5]{ruzhansky-turunen} and here we write some useful consequences.

\begin{theorem}
    Let $m \in \mathbb{R}$, let $0\leq\delta<1$, and let $0< \rho \leq 1$. The symbol $\sigma \in S^m_{\rho,\delta} (\mathbb{T}^n\times \mathbb{Z}^n) $  if and only if there exists a symbol $\tilde{\sigma} \in S^m_{\rho,\delta} (\mathbb{T}^n\times \mathbb{R}^n) $ such that $\sigma = \tilde{\sigma}|_{\mathbb{T}^n\times \mathbb{Z}^n} $. Moreover, this extension is unique modulo $S^{-\infty}(\mathbb{T}^n\times \mathbb{R}^n)$.
\end{theorem}

\begin{remark}
    When we employ this extension, the definition of its corresponding operator may be adjusted to 
    \begin{equation*}
        T_{\tilde{\sigma}}f(x) = \int_{\mathbb{T}^n} \left[ \int_{\mathbb{R}^n} e^{i2\pi (x-y) \cdot \xi} \tilde{\sigma}(x, \xi)\diff \xi \right] f(y) \diff y ,
    \end{equation*}
    so that the Schwartz kernel of the operator, in the sense of distributions, is defined as the integral 
    \begin{equation*}
        \tilde{k}(x, y) = \int_{\mathbb{R}^n} e^{i2\pi (x-y) \cdot \xi} \tilde{\sigma}(x, \xi)\diff \xi.
    \end{equation*}
    This kernel representation allows us to use techniques such as integration by parts when discussing properties of the Schwartz kernel.
\end{remark}
Hence, there is a correspondence between toroidal symbols with discrete and continuous frequency domains. This property can be extended to the corresponding operators as well.
\begin{theorem}[Equivalence of operator classes]
    Let $m \in \mathbb{R}$, let $0\leq\delta<1$, and let $0< \rho \leq 1$. Then 
        \begin{equation*}
            \Psi^m_{\rho,\delta}(\mathbb{T}^n\times \mathbb{R}^n) = \Psi^m_{\rho,\delta}(\mathbb{T}^n\times \mathbb{Z}^n).
        \end{equation*}
\end{theorem}
Now, we define the maximal operators that will allow us to prove the boundedness of toroidal pseudo-differential operators.
\begin{definition}[Sharp maximal operator]
    For an integrable function $f$ defined on $\mathbb{T}^n$, we define its corresponding \textit{$r$-sharp maximal function} by 
    \begin{equation*}
        \mathcal{M}^\#_rf(x) := \sup_{Q:x\in Q} \inf_{c_Q \in \mathbb{C}}\left(
        \frac{1}{|Q|} \int_Q |f(y) - c_Q|^r\diff y
        \right)^{1/r},
    \end{equation*}
    where the supremum is taken over all cubes in $\mathbb{T}^n$ with sides parallel to the axes.
\end{definition}
In \cite{park-tomita}, the authors work with the inhomogeneous version of the sharp maximal operator. This alternative definition makes a distinction between cubes with side length greater than one and less than one. In the case of the torus, this distinction is not necessary and in this paper we employ the definition stated above.
\begin{definition}[Hardy-Littlewood maximal operator]
    For an integrable function $f$ defined on $\mathbb{T}^n$, we define the \textit{r-Hardy-Littlewood maximal operator} by 
    \begin{equation*}
        \mathrm{M}_rf(x) := \sup_{Q:x\in Q} \left( 
        \frac{1}{|Q|} \int_Q |f(y)|^r \diff y
        \right)^{1/r}
    \end{equation*}
    where the supremum is taken over all cubes in $\mathbb{T}^n$ with sides parallel to the axes.
\end{definition}

\begin{definition}[Weighted Lebesgue spaces $L^p(w)$]
We say that a nonnegative integrable function $w$ defined on $\mathbb{T}^n$ belongs to \textit{Muckenhoupt's $A_p$ class of weights} if it satisfies the following properties
\begin{align*}
    \mathrm{M}w(x) \lesssim w(x) \quad \text{for almost all }x\in\mathbb{T}^n, & \quad \quad \text{when } p=1, \\
    \sup_{Q:cubes} \left(
    \frac{1}{|Q|} \int_Qw(x)\diff x
    \right) \left( 
    \frac{1}{|Q|} \int_Q (w(x))^{-1/(p-1)}
    \right)^{p-1} < \infty,
    & \quad \quad \text{when } 1<p<\infty.
\end{align*}
For a weight $w$, the \textit{weighted Lebesgue space} $L^p(\mathbb{T}^n; w)$ with $1\leq p<\infty$, consists of all measurable functions $f$ on $\mathbb{T}^n$ satisfying
\begin{equation*}
    \|f\|_{L^p(\mathbb{T}^n;w)} := \left( 
    \int_{\mathbb{T}^n} |f(x)|^pw(x)\diff x
    \right)^{1/p} < \infty.
\end{equation*}
In order to simplify the notation, in general we will use $L^p(w)$ and $\|\cdot\|_{L^p(w)}$, and when the weight is the constant function equal to one, we will use $L^p$ and $\|\cdot\|_p$ during the remainder of this paper.
\end{definition}
As it was proved in \cite{muckenhoupt} for $\mathbb{R}^n$, and for arbitrary measure spaces in \cite[Section~IV.1]{garcia-defrancia}, we have that for $1<p<\infty$, we have that 
\begin{equation}
    w \in A_p \quad \text{if and only if} \quad \|\mathrm{M}f\|_{L^p(w)} \lesssim\|f\|_{L^p(w)}.
    \label{eq:muckenhoupt}
\end{equation}
For $p=\infty$ we define $A_\infty := \bigcup_{p>1} A_p$. Moreover, we have that $A_p \subset A_q$ whenever $1\leq p \leq q \leq \infty$. Now, we define function spaces that establish certain regularity conditions on their elements, making them useful in the estudy of PDE's.
\begin{definition}[Sobolev spaces $W^s_p$]
    For $s\in\mathbb{R}$ and $1\leq p \leq \infty$, we say a 1-periodic function $f$ is in the \textit{Sobolev space} $W^s_p(\mathbb{T}^n)$ if 
    \begin{equation*}
        \|f\|_{W^s_p} := \|J^sf\|_p < \infty,
    \end{equation*}
    where $J^s$ is the Bessel potential of order $s$.
\end{definition}
In order to define a more general scale than the one of Sobolev spaces, the Besov spaces, let $\varphi\in\mathcal{S}(\mathbb{R}^n)$ be such that
    \begin{equation*}
        \supp \varphi \subset \{\xi \in \mathbb{R}^n : 2^{-1} \leq |\xi| \leq 2\}, 
    \end{equation*}
    \begin{equation*}
        \varphi(\xi) > 0, \quad \text{ for } \quad 2^{-1}<|\xi| < 2,
    \end{equation*}
    \begin{equation*}
        \sum_{k=-\infty}^\infty \varphi(2^k\xi) = 1, \quad \text{ when } \quad \xi \neq 0.
    \end{equation*}
\begin{definition}[Besov spaces $B^s_{p,q}$]
    First, let us define for $k \in \mathbb{Z}$, the functions
    \begin{equation*}
        \mathcal{F}\varphi_k(\xi) := \varphi(2^k\xi), \quad \text{ and } \quad \mathcal{F}\psi(\xi)  := 1 - \sum_{k=1}^\infty \varphi(2^{-k}\xi),
    \end{equation*}
    where $\varphi$ is as above. Then, we say that a 1-periodic function $f$ belongs to the Besov space $B^s_{p,q}(\mathbb{T}^n)$ if 
    \begin{equation*}
        \|f\|_{B^s_{p,q}} := \|\psi * f\|_p + \left( \sum_{k=1}^\infty (2^{sk} \|\varphi_k*f \|_p)^q
        \right)^{1/q} < \infty.
    \end{equation*}
\end{definition}
It is known from \cite[Theorem~6.2.4]{bergh-lofstrom} that Besov spaces can be obtained from the real interpolation of Sobolev spaces, namely that,
\begin{equation*}
    (W^{s_0}_p, W^{s_1}_p)_{\theta, q} = B^s_{p,q},
\end{equation*}
for $s = \theta s_0 + (1-\theta)s_1$. Here we have used $(\cdot,\cdot)_{\theta,q}$ for the real interpolation functor.

\section{Main results}
\label{section:main}
\subsection{Kernel estimates}
\label{ssection:kernel}
Let $\varphi$ be a Schwartz function defined on $\mathbb{R}^n$ such that its Fourier transform $\widehat{\varphi}$ is equal to one on the unit ball centered at the origin and is supported in the concentric ball of radius 2. Let $\psi$ be another test function such that $\widehat{\psi}(\xi) := \widehat{\varphi}(\xi) - \widehat{\varphi}(2\xi)$ for $\xi \in \mathbb{R}^n$. For each $k \in \mathbb{N}$, we define
\begin{equation*}
    \psi_k(x):=2^{kn}\psi(2^kx).
\end{equation*}
Then we have an inhomogeneous Littlewood-Paley partition of unity formed by $\varphi$ and $\psi_k$, with $k\in\mathbb{N}$. Moreover, notice that 
\begin{equation*}
    \supp \widehat{\psi_k} \subset \left\{\xi \in \mathbb{R}^n : 2^{k-1} \leq |\xi|\leq 2^{k+1}\right\}
\end{equation*}
and that
\begin{equation*}
    \widehat{\varphi}(\xi) + \sum_{k\in\mathbb{N}} \widehat{\psi_k}(\xi) = 1.
\end{equation*}
Thus, we can decompose any $\sigma \in S^m_{\rho,\rho}(\mathbb{T}^n \times \mathbb{Z}^n)$ as 
\begin{equation*}
    \sigma(x,\xi) = \sigma(x,\xi)\widehat{\varphi}(\xi) + \sum_{k\in \mathbb{N}} \sigma(x,\xi)\widehat{\psi_k}(\xi) =: \sigma_0(x,\xi) + \sum_{k\in \mathbb{N}} \sigma_k(x,\xi),
\end{equation*}
so we can write 
\begin{equation*}
    T_\sigma = \sum_{k\in \mathbb{N}_0} T_{\sigma_k},
\end{equation*}
where $T_{\sigma_k}$ are the toroidal pseudo-differential operators associated with $\sigma_k \in S^m_{\rho,\rho}(\mathbb{T}^n \times \mathbb{Z}^n)$. Now, we write their corresponding kernels (in the sense of distributions discussed in the previous section) as
\begin{equation}
    K_k(y, u) := \sum_{\xi \in \mathbb{Z}^n} \sigma_k(y, \xi)e^{i2\pi u \cdot \xi}.
    \label{eq:kernel-Kk}
\end{equation}
Moreover, let $\tilde{\sigma}_k$ be the extension of the symbol $\sigma_k$ defined on $\mathbb{T}^n \times \mathbb{R}^n$. First, we prove the following kernel estimates.
\begin{lemma}
    Let $0 < \rho \leq 1$ and $m\in \mathbb{R}$. Suppose that $\sigma \in S^m_{\rho,\rho}(\mathbb{T}^n \times \mathbb{Z}^n) $ and let $K_k$ be defined as in \cref{eq:kernel-Kk}. For arbitrary $N\geq 0$ and $1\leq r\leq 2$, we have that 
    \begin{equation*}
        \left\| (1+2^{k\rho}|u|)^N K_k(y, u)
        \right\|_{r'} \lesssim_N 2^{k(m+n/r)},
    \end{equation*}
    \begin{equation*}
        \left\| (1+2^{k\rho}|u|)^N \nabla_y K_k(y, u)
        \right\|_{r'} \lesssim_N 2^{k(\rho+m+n/r)},
    \end{equation*}
    \begin{equation*}
        \left\| (1+2^{k\rho}|u|)^N \nabla_u K_k(y, u)
        \right\|_{r'} \lesssim_N 2^{k(1+m+n/r)},
    \end{equation*}
    as functions of $u\in\mathbb{R}^n$ uniformly in $y\in\mathbb{R}^n$.
    \label{lem:kernel-estimates}
\end{lemma}
\begin{proof}
    We follow the argument of Park and Tomita \cite{park-tomita}. Since $\tilde{\sigma}_k \in S^m_{\rho,\rho}(\mathbb{T}^n \times \mathbb{R}^n) $, then for any multi-index $\beta\in \mathbb{N}_0^n$ we have that
    \begin{equation}
        \left|\partial_\xi^\beta \tilde{\sigma}_k(y, \xi)\right| \lesssim 2^{k(m-\rho|\beta|)},
        \label{eq:partial-sigma}
    \end{equation}
    \begin{equation}
        \left|\partial_\xi^\beta \nabla_y \tilde{\sigma}_k(y, \xi)\right| \lesssim 2^{k(m+\rho-\rho|\beta|)},
        \label{eq:nabla-sigma}
    \end{equation}
    \begin{equation}
        \left|\partial_\xi^\beta [\xi \  \tilde{\sigma}_k(y, \xi)]\right| \leq 
        \left|\xi\cdot\partial_\xi^\beta \tilde{\sigma}_k(y,\xi)\right| + \sum_{j=1}^n \left| \partial_\xi^{\beta-j} \tilde{\sigma}_k(y,\xi) \right|
        \lesssim
         2^{k(m+1-\rho|\beta|)},
        \label{eq:xi-sigma}
    \end{equation}
        
    since $\langle\xi\rangle \sim 2^k$. Now, we use Hausdorff-Young's inequality and \cref{eq:partial-sigma} to obtain that
    \begin{align*}
        \left\| (2^{k\rho}u)^\beta K_k(y, u)
        \right\|_{r'} & \lesssim2^{k\rho|\beta|} \left\| 
        \partial_\xi^\beta \tilde{\sigma}_k(x, \xi)
        \right\|_{r} \\
        & = 2^{k\rho|\beta|} 2^{k(m-\rho|\beta|)} 2^{kn/r} \\
        & =  2^{k(m+n/r)},
    \end{align*}
    since the volume of the support of $\tilde{\sigma}_k$ is comparable with $2^{kn}$. Moreover, here the $r'$-norm is taken with respect to $u \in \mathbb{R}^n$ and the $r$-norm is taken with respect to $\xi\in\mathbb{R}^n$. This concludes the proof of the first estimate, the remaining estimates can be proved using the same procedure and the estimates \cref{eq:nabla-sigma} and \cref{eq:xi-sigma} respectively.
\end{proof}
\subsection{Boundedness results}
\label{ssection:boundedness}
Now, let us recall the following $L^p$-$L^q$ continuity result for toroidal pseudo-differential operators.
\begin{theorem}
    \cite[Theorem~3.12]{Cardona:Martinez} Let $0\leq\delta<1$, let $0<\rho\leq 1$, let $m\in\mathbb{R}$, and let $\sigma\in S^m_{\rho,\delta}(\mathbb{T}^n\times\mathbb{Z}^n)$. Then, $T_\sigma$ extends to a bounded operator from $L^p$ into $L^q $ where $1<p\leq q<\infty$ when,  
    \begin{enumerate}
        \item $1<p\leq 2 \leq q$ and 
        \begin{equation}
             m \leq - n \left( \frac{1}{p} - \frac{1}{q} + \lambda
            \right),
            \label{eq:p2q}
        \end{equation}
        \item if $2 \leq p \leq q$ and 
        \begin{equation}
             m \leq - n \left[ \frac{1}{p} - \frac{1}{q} + (1-\rho) \left( \frac{1}{2} - \frac{1}{p}
            \right)
            + \lambda
            \right],
        \end{equation}
        \item if $p\leq q \leq 2$ and 
        \begin{equation}
             m \leq - n \left[ \frac{1}{p} - \frac{1}{q} + (1-\rho) \left( \frac{1}{q} - \frac{1}{2}
            \right)
            + \lambda
            \right],
        \end{equation}
    \end{enumerate}
     where $\lambda := \max\{ 0, (\delta-\rho)/2 \}$. 
     \label{theo:Lp-Lq}
\end{theorem}
Now, we prove some useful boundedness lemmas.
\begin{lemma}
    Let $1 < r \leq 2$, let $0<\rho\leq r/2$, let $0<\rho<1$, and let $m = -n(1-\rho)/r$. Then for every $\sigma \in S^m_{\rho,\rho}(\mathbb{T}^n \times \mathbb{Z}^n) $, its corresponding operator $T_\sigma$ is continuous from $L^r$ into $L^{r/\rho}$.
    \label{lem:Lr-Lr/rho-boundedness}
\end{lemma}
\begin{proof}
    Notice that we have $r\leq 2 \leq r/\rho$ and $m$ satisfies \cref{eq:p2q}.
\end{proof}
In order to handle the case $\rho > r/2$ that is not considered in the previous lemma, we prove the following.
\begin{lemma}
    Let $1<r<2$, let $r/2\leq\rho<1$, and let $m=-n(1-\rho)/r$. Suppose that $k\in\mathbb{N}_0$ and 
    \begin{equation*}
        \frac{2\rho-r}{2-r} < \lambda < \rho.
    \end{equation*}
    Then every $\sigma \in S^m_{\rho,\rho}(\mathbb{T}^n \times \mathbb{Z}^n) $ satisfies 
    \begin{equation*}
        \|T_{\sigma_k}f\|_{\frac{r(1-\lambda)}{\rho-\lambda} } \lesssim 2^{\lambda nk \frac{1-\rho}{r(1-\lambda)}} \|f\|_r
    \end{equation*}
    for $f\in C^\infty(\mathbb{T}^n)$.
    \label{lem:Lr-lambda-boundedness}
\end{lemma}

\begin{proof}
    Let us use the strategy by Park and Tomita \cite{park-tomita}. Let us define    
    \begin{equation*}
        a_k(x, \xi) := \tilde{\sigma}_k(2^{-\lambda k}x, 2^{\lambda k} \xi).
    \end{equation*}
    Then for any multi-indices $\alpha, \beta \in \mathbb{N}_0^n$, we obtain that
    \begin{align*}
        |\partial_x^\alpha\partial_\xi^\beta a_k(x, \xi)| & = 2^{\lambda k(|\beta|-|\alpha|)} |\partial_x^\alpha\partial_\xi^\beta \tilde{\sigma}_k(2^{-\lambda k}x, 2^{\lambda k} \xi)| \\
        & \lesssim 2^{\lambda k(|\beta|-|\alpha|)} 2^{k(m - \rho(|\beta|-|\alpha|))} \chi_{ \{ \langle\xi\rangle \sim 2^{(1-\lambda)k} \} } \\
        & =  2^{ k[m - (\rho-\lambda)(|\beta|-|\alpha|)] } \chi_{ \{ \langle\xi\rangle \sim 2^{(1-\lambda)k} \} } \\
        & \lesssim  \langle\xi\rangle^{ \frac{m}{1-\lambda}  -  \frac{\rho-\lambda}{1-\lambda}(|\beta|-|\alpha|) },
    \end{align*}
    where the constant is independent of $k$. Therefore, we have that 
    \begin{equation*}
        a_k \in S^{\frac{m}{1-\lambda}}_{ \frac{\rho-\lambda}{1-\lambda} , \frac{\rho-\lambda}{1-\lambda} } (\mathbb{T}^n\times\mathbb{R}^n)
    \end{equation*}
    uniformly in $k$. Notice that 
    \begin{equation*}
        0 < \frac{\rho-\lambda}{1-\lambda} < \frac{r}{2} \quad \text{ and } \quad \frac{m}{1-\lambda} = -\frac{n}{r} \left(  1 - \frac{\rho-\lambda}{1-\lambda}
        \right),
    \end{equation*}
    which allow us to employ \cref{lem:Lr-Lr/rho-boundedness}, and obtain the $L^r$-$L^\frac{r(1-\lambda)}{\rho-\lambda}$-continuity for the operators $T_{a_k}$ uniformly in $k$. Since we have that 
    \begin{equation*}
        T_{\sigma_k}f(x) = T_{a_k}(f(2^{-\lambda k}\cdot ))(2^{
        \lambda k
        }x),
    \end{equation*}
    we can obtain that 
    \begin{align*}
        \|T_{\sigma_k}f\|_\frac{r(1-\lambda)}{\rho-\lambda} &= 2^{ -\lambda nk \frac{\rho-\lambda}{r(1-\lambda)} } \|T_{a_k}(f(2^{-\lambda k} \cdot))\|_\frac{r(1-\lambda)}{\rho-\lambda} \\
        & \lesssim 2^{ -\lambda nk \frac{\rho-\lambda}{r(1-\lambda)} } \|f(2^{-\lambda k} \cdot)\|_r \\
        & =  2^{ -\lambda nk \frac{\rho-\lambda}{r(1-\lambda)} } 2^{\lambda nk/r} \|f\|_r \\
        & =  2^{ \lambda nk \frac{1-\rho}{r(1-\lambda)} } \|f\|_r.
    \end{align*}
    Thus, completing the proof.
\end{proof}

Now, let $P$ be a concentric dilate of $Q$ with $\ell(P)\geq 10\sqrt{n}\ell(Q)$ and let $\chi_P$ denote its characteristic function. Now, let us consider 
\begin{equation}
    f = f_{\chi_P} + f_{\chi_{\mathbb{T}^n\setminus P}} =: f_0 + f_1,
    \label{eq:decomposition}
\end{equation}
so that
\begin{equation*}
    T_{\sigma_k}f = T_{\sigma_k}f_0 + T_{\sigma_k}f_1.
\end{equation*}

\begin{proposition}
    Let $0<\rho<1$, let $1<r\leq2$, and let $m=-n(1-\rho)/r$. Suppose that $x\in Q$.

    \begin{enumerate}
        \item Let $0<\rho<r/2$, and let $k\in\mathbb{N}_0$. Then every symbol $\sigma \in S^m_{\rho,\rho}(\mathbb{T}^n \times \mathbb{Z}^n) $ satisfies
        \begin{equation}
            \left( \frac{1}{|Q|}\int_Q |T_{\sigma_k} f_0(y)|^r\diff y  
            \right)^{1/r} \lesssim \left[ \frac{\ell(P)}{\ell(Q)^\rho} 
            \right]^{n/r} \mathrm{M}_rf(x).
            \label{eq:fo-estimate-rho-less}
        \end{equation}
        \item Let $r/2\leq \rho <1$, let $ \frac{2\rho - r}{2-r} < \lambda < \rho$, and let $k\in\mathbb{N}_0$. Then every symbol $\sigma \in S^m_{\rho,\rho}(\mathbb{T}^n \times \mathbb{Z}^n) $ satisfies
        \begin{equation}
            \left( \frac{1}{|Q|}\int_Q|T_{\sigma_k}f_0(y)|^r \diff y
            \right)^{1/r} \lesssim [2^k\ell(Q)]^{\lambda n \frac{1-\rho}{r(1-\lambda)}} \left[ \frac{\ell(P)}{\ell(Q)^\rho}
            \right]^{n/r} \mathrm{M}_rf(x).
            \label{eq:f0-estimate-rho-greater}
        \end{equation}
    \end{enumerate}

\end{proposition}
\begin{proof}
    First, let us consider the case $0<\rho<r/2$. By H\"older's inequality and \cref{lem:Lr-Lr/rho-boundedness} we have that 
    \begin{align*}
        \left(
        \frac{1}{|Q|} \int_Q |T_{\sigma_k}f_0(y)|\diff y
        \right)^{1/r} &\leq \frac{1}{|Q|^{\rho/r}} \left\| T_{\sigma_k}f_0  \right\|_{r/\rho} \\
        & \lesssim \frac{1}{\ell(Q)^{n\rho/r}} \|f \chi_P \|_r \\
        & \lesssim \left[ \frac{\ell(P)}{\ell(Q)^\rho} 
            \right]^{n/r} \mathrm{M}_rf(x),
    \end{align*}
    since $x \in Q \subset P$. Hence, proving \cref{eq:fo-estimate-rho-less}. Now, let us assume that $r/2\leq \rho<1$. By H\"older's inequality and \cref{lem:Lr-lambda-boundedness} we obtain that 
    \begin{align*}
        \left(
        \frac{1}{|Q|} \int_Q |T_{\sigma_k}f_0(y)|\diff y
        \right)^{1/r} &\leq {|Q|^{-\frac{\rho-\lambda}{r(1-\lambda)}}} \|T_{\sigma_k}f_0\|_\frac{r(1-\lambda)}{\rho-\lambda} \\
        &\lesssim {\ell(Q)^{-n\frac{\rho-\lambda}{r(1-\lambda)}}} 2^{\lambda nk\frac{1-\rho}{r(1-\lambda)}} \|f\chi_P\|_r  \\
        & \lesssim 2^{\lambda nk\frac{1-\rho}{r(1-\lambda)}}\ell(Q)^{-n\frac{\rho-\lambda}{r(1-\lambda)}} \ell(P)^{n/r} \mathrm{M}_rf(x)\\
        &\lesssim[2^k\ell(Q)]^{\lambda n\frac{1-\rho}{n(1-\lambda)}} 
        \left[ \frac{\ell(P)}{\ell(Q)^\rho}
        \right]^{n/r} \mathrm{M}_rf(x).
    \end{align*}
    Thus, completing the proof.
\end{proof}
We now prove estimates for the second part of the decomposition as in \cref{eq:decomposition}.
\begin{proposition}
    Let $0\leq \rho<1$, let $1\leq r\leq2$, let $m=-n(1-\rho)/r$, and $k\in\mathbb{N}_0$. Suppose that $x, y \in Q$. Then every $\sigma \in S^m_{\rho,\rho}(\mathbb{T}^n \times \mathbb{Z}^n) $ satisfies
    \begin{equation}
        |T_{\sigma_k}f_1(y)| \lesssim_N [2^{k\rho}\ell(P)]^{-(N-n/r)} \mathrm{M}_rf(x),
        \label{eq:f1-estimate}
    \end{equation}
    and
    \begin{equation}
        |T_{\sigma_k}f_1(y) - T_{\sigma_k}f_1(x)| \lesssim_N 2^k\ell(Q) [2^{k\rho}\ell(P)]^{-(N-n/r)} \mathrm{M}_rf(x),
        \label{eq:diff-f1}
    \end{equation}
    for any $N>n/r$.
\end{proposition}
\begin{proof}
    First, we consider \cref{eq:f1-estimate}. Let $N>n/r$, and let $y\in Q$. By H\"older's inequality we obtain that
    \begin{equation*}
        |T_{\sigma_k}f_1(y)| \leq \int_{\mathbb{T}^n\setminus P} |K_k(y, y-u)||f(u)|\diff u 
    \end{equation*}
    \begin{equation*}
        \leq \left\| \ell(P)^{n/r} \left[ \frac{|\cdot|}{\ell(P)} \right]^N K_k(y, \cdot) 
        \right\|_{r'} \left\| \ell(P)^{-n/r} \left[ \frac{|y-\cdot|}{\ell(P)} \right]^{-N} f\chi_{\mathbb{T}^n\setminus P} 
        \right\|_r.
    \end{equation*}
    The $L^{r'}$-norm can be estimated using \cref{lem:kernel-estimates} by 
    \begin{align*}
        \left\| \ell(P)^{n/r} \left[ \frac{|\cdot|}{\ell(P)} \right]^N K_k(y, \cdot) 
        \right\|_{r'} & = \ell(P)^{-(N-n/r)} \left\| |\cdot|^N K_k(y, \cdot)  \right\|_{r'} \\
        & \leq \ell(P)^{-(N-n/r)} 2^{-k\rho N} \left\| (1+2^{k\rho} |\cdot|)^N K_k(y, \cdot)  \right\|_{r'} \\
        & \lesssim \ell(P)^{-(N-n/r)} 2^{-k\rho N} 2^{k(m+n/r)}\\
        & = [2^{k\rho}\ell(P)]^{-(N-n/r)}.
    \end{align*}
    On the other hand, the $L^r$-norm can be estimated by noting that for $x, y \in Q$, and for $u \in \mathbb{T}^n\setminus P$, we have that
    \begin{equation*}
        |y-u|\geq |x-u|-|x-y| \geq C_n[\ell(P)+|x-u|].
    \end{equation*}
    Thus, for $N>n/r$, we obtain that
    \begin{align*}
        \left\| \ell(P)^{-n/r} \left[ \frac{|y-\cdot|}{\ell(P)} \right]^{-N}f \chi_{\mathbb{T}^n\setminus P}
        \right\|_r& \lesssim\left( \int_{\mathbb{T}^n} \ell(P)^{-n} \left[ 1+\frac{|x-u|}{\ell(P)}
        \right]^{-rN} |f(u)|^r \diff u
        \right)^{1/r} \\
        & \lesssim\mathrm{M}_rf(x).
    \end{align*}
    Thus, we combine both estimates to obtain \cref{eq:f1-estimate}. Now, we consider \cref{eq:diff-f1} using a similar proof as in the previous case, now defining 
    \begin{equation*}
        H_k(y, x, u) := K_k(y, y-u) - K_k(x, x-u).
    \end{equation*}
    Assume $x, y \in Q$ and $N>n/r$. Using H\"older's inequality, we have that
    \begin{equation*}
        |T_{\sigma_k}f_1(y) - T_{\sigma_k}f_1(x)| \leq \int_{\mathbb{T}^n\setminus P} |H_k(y, x, u)||f(y)|\diff u
    \end{equation*}
    \begin{equation*}
        \leq \left\| \ell(P)^{n/r} \left[ \frac{|y-\cdot|}{\ell(P)}
        \right]^N H_k(y, x, \cdot) \chi_{\mathbb{T}^n\setminus P}
        \right\|_{r'} \left\| \ell(P)^{-n/r} \left[ \frac{|y-\cdot|}{\ell(P)}
        \right]^{-N} f\chi_{\mathbb{T}^n\setminus P}
        \right\|_r.
    \end{equation*}
    Let us estimate the first factor using \cref{lem:kernel-estimates} by
    \begin{equation*}
        \left\| \ell(P)^{n/r} \left[ \frac{|y-\cdot|}{\ell(P)}
        \right]^N H_k(y, x, \cdot) \chi_{\mathbb{T}^n\setminus P}
        \right\|_{r'}
    \end{equation*}
    \begin{equation*}
        \lesssim \left\| \ell(Q)\ell(P)^{n/r} \left[ \frac{|y-\cdot|}{\ell(P)}
        \right]^N \int_0^1|\nabla K_k(y(t), y(t) - \cdot)|\diff t \chi_{\mathbb{T}^n\setminus P}
        \right\|_{r'}
    \end{equation*}
    \begin{equation*}
        \lesssim \ell(Q)\ell(P)^{-(N-n/r)} \int_0^1\left\| |y(t) - \cdot|^N |\nabla K_k(y(t), y(t) - \cdot)|
        \right\|_{r'} \diff t
    \end{equation*}
    \begin{equation*}
        \lesssim \ell(Q)\ell(P)^{-(N-n/r)} 2^{-k\rho N} \int_0^1 \left\| (1+2^{k\rho}|\cdot|)^N |\nabla K_k(y(t), \cdot)| 
        \right\|_{r'} \diff t
    \end{equation*}
    \begin{equation*}
        \lesssim \ell(Q)\ell(P)^{-(N-n/r)} 2^{-k\rho N} 2^{k(1+m+n/r)} = C 2^k \ell(Q)[2^{k\rho}\ell(P)]^{-(N-n/r)},
    \end{equation*}
    where $y(t) := ty + (1-t)x \in Q$ so that $|y-x|\lesssim|y(t)-u|$ for $u \in \mathbb{T}^n\setminus P$. For the second factor, we use the same estimate as in the previous case, completing the proof.
\end{proof}
Let us proceed to prove the main theorem of this paper, from which the desired continuity will follow.
\begin{theorem}
    Let $1<r\leq 2$, and let $0<\rho<1$. Suppose that $m\leq-n(1-\rho)/r$ and $\sigma \in S^m_{\rho,\rho}(\mathbb{T}^n \times \mathbb{Z}^n)  $. Then we it holds that 
    \begin{equation*}
        \mathcal{M}^\#_r(T_\sigma f)(x) \lesssim\mathrm{M}_r f(x)
    \end{equation*}
    for $f \in C^\infty(\mathbb{T}^n)$.
    \label{theo:sharp-maximal}
\end{theorem}
\begin{proof}
    Notice that is suffices to prove 
    \begin{equation}
        \inf_{c_Q\in \mathbb{C}} \left( 
        \frac{1}{|Q|} \int_Q|T_\sigma f(y) - c_Q|^r\diff y
        \right)^{1/r} \lesssim \mathrm{M}_rf(x) ,
        \label{eq:main-goal}
    \end{equation}
    uniformly in $Q$ and $x\in \mathbb{T}^n$. Now, let $P_\rho$ be a concentric dilation of $Q$ so that $\ell(P_\rho) = 10\sqrt{n}\ell(Q)^\rho$ and let us decompose $f$ as in \cref{eq:decomposition}. We will first consider the case when $0<\rho < r/2$. Then, the left hand side of \cref{eq:main-goal} is less than the sum of 
    \begin{equation*}
        \mathcal{I}_0 := \left( 
        \frac{1}{|Q|} \int_Q|T_\sigma f_0(y)|^r\diff y
        \right)^{1/r}
    \end{equation*}
    and 
    \begin{equation*}
        \mathcal{I}_1 := \inf_{c_Q \in \mathbb{C}} \left(  \frac{1}{|Q|}\int_Q|T_\sigma f_1 - c_Q|^r
        \right)^{1/r}.
    \end{equation*}
    Using H\"older's inequality and \cref{lem:Lr-Lr/rho-boundedness} we can verifiy that 
    \begin{equation*}
        \mathcal{I}_0 \leq \frac{1}{|Q|^{\rho/r}} \|T_\sigma f_0\|_{r/\rho} \lesssim\frac{1}{\ell(Q)^{n\rho/r}} \|f\chi_{P_\rho}\|_r \lesssim\mathrm{M}_rf(x). 
    \end{equation*}
    In order to estimate $\mathcal{I}_1$, let us set 
    \begin{equation}
        c_Q := \sum_{k:2^k\ell(Q)<1} T_{\sigma_k} f_1(x), 
        \label{eq:cQ}
    \end{equation}
    so we have that
    \begin{equation*}
        |T_\sigma f_1(y) - c_Q| \leq \sum_{k:2^k\ell(Q) \geq 1} |T_{\sigma_k}f_1(y)| + \sum_{k:2^k\ell(Q)<1}|T_{\sigma_k} f_1(y) - T_{\sigma_k} f_1(x)|  .
    \end{equation*}
    First, we estimate the first sum using \cref{eq:f1-estimate} by
    \begin{equation*}
        \sum_{k:2^k\ell(Q) \geq 1} |T_{\sigma_k}f_1(y)| \lesssim \sum_{k:2^k\ell(Q) \geq 1} [2^k\ell(Q)]^{-\rho(N-n/r)} \mathrm{M}_rf(x) \lesssim \mathrm{M}_rf(x),
    \end{equation*}
    as $\rho > 0$ and $N > n/r$. Then, we apply \cref{eq:diff-f1} to estimate the remaining terms by 
    \begin{equation*}
        \sum_{k:2^k\ell(Q)<1}|T_{\sigma_k} f_1(y) - T_{\sigma_k} f_1(x)| \lesssim_N\sum_{k:2^k\ell(Q)<1} [2^k\ell(Q)]^{1-\rho(N-n/r)}\mathrm{M}_r f(x) = C\mathrm{M}_rf(x),
    \end{equation*}
    when we choose $N<n/r + 1/\rho$. Thus, completing the proof when $0<\rho<r/2 $. Now, let us consider the case $r/2\leq \rho<1$. We can estimate \cref{eq:main-goal} by the sum of 
    \begin{equation*}
        \mathcal{J}_1 := \inf_{c_Q \in \mathbb{C}} \left( \frac{1}{|Q|}\int_Q \left| \sum_{k:2^k\ell(Q)<1} T_{\sigma_k} f(y) - c_Q \right|^r \diff y
         \right)^{1/r},
    \end{equation*}
    \begin{equation*}
        \mathcal{J}_2 := \sum_{\substack{k:2^k\ell(Q)\geq 1, \\ 2^{\rho k}\ell(Q) <1}} \left(\frac{1}{|Q|} \int_Q |T_{\sigma_k}f(y)|^r \diff y 
        \right)^{1/r},
    \end{equation*}
    \begin{equation*}
        \mathcal{J}_3 := \sum_{k:2^{\rho k} \ell(Q) \geq 1} \left( \frac{1}{|Q|} \int_Q |T_{\sigma_k}f(y)|^r \diff y 
        \right)^{1/r}.
    \end{equation*}
    First, we set $\frac{2\rho - r}{2-r} < \lambda < \rho$ for the remainder of this proof. Then, we use the same decomposition by $P_\rho$ as above and use $c_Q$ as defined in \cref{eq:cQ} in order to estimate $\mathcal{J}_1$ by
    \begin{equation*}
        \sum_{k:2^k\ell(Q)<1} \left[
        \left(  \frac{1}{|Q|} \int_Q |T_{\sigma_k} f_0 (y)|^r \diff y
        \right)^{1/r} + \left( \frac{1}{|Q|}\int_Q|T_{\sigma_k} f_1(y) -T_{\sigma_k}f_1(x)|^r \diff y
        \right)^{1/r}
        \right].
    \end{equation*}
    Using \cref{eq:f0-estimate-rho-greater} we obtain that 
    \begin{equation*}
        \left( 
        \frac{1}{|Q|} \int_Q |T_{\sigma_k} f_0 (y)|^r \diff y
        \right)^{1/r} \lesssim [2^k\ell(Q)]^{\lambda n \frac{1-\rho}{r(1-\lambda)}} \mathrm{M}_rf(x).
    \end{equation*}
    Moreover, we employ \cref{eq:diff-f1} to get that
    \begin{equation*}
        |T_{\sigma_k} f_1(y) - T_{\sigma_k} f_1(x)| \lesssim [2^k\ell(Q)]^{1-\rho(N-n/r)} \mathrm{M}_rf(x).
    \end{equation*}
    Again, we choose $N < n/r + 1/\rho$, to obtain an uniform bound and complete the proof for the first term. For the term $\mathcal{J}_2$, we choose a positive number $\varepsilon$ such that 
    \begin{equation*}
        \lambda \left( \frac{1-\rho}{1-\lambda} \right) < \varepsilon < \rho,
    \end{equation*}
    which we know exists since $\lambda < \rho$. Now, let us define $P_{\varepsilon, k}$ as the concentric dilate of $Q$  with $\ell(P_{\varepsilon, k}) = 10 \sqrt{n}\ell(Q)^\rho [2^k\ell(Q)]^{-\varepsilon} $. Let us note that 
    \begin{equation*}
        10\sqrt{n} \ell(Q) \leq \ell(P_{\varepsilon, k}),
    \end{equation*}
    since it is equivalent to 
    \begin{equation*}
        2^{\rho k} \ell(Q)^{\frac{\rho(1-\rho + \varepsilon)}{\varepsilon} } \leq 1, 
    \end{equation*}
    which is true as 
    \begin{align*}
        \rho(1-\rho) & \geq \varepsilon(1-\rho) \\
        \frac{\rho(1-\rho + \varepsilon)}{\varepsilon} & \geq 1,
    \end{align*}
    and $2^{\rho k}\ell(Q) < 1$. Then, we consider 
    \begin{equation*}
        f = f\chi_{P_{\varepsilon, k}} + f\chi_{ \mathbb{T}^n\setminus P_{\varepsilon, k}} =: f_{0,k} + f_{1,k}.
    \end{equation*}
    Hence, by \cref{eq:f0-estimate-rho-greater} we have that 
    \begin{align*}
        \left( \frac{1}{|Q|}\int_Q |T_{\sigma_k} f_{0, k}(y)|^r \diff y 
        \right)^{1/r} & \lesssim [2^k\ell(Q)]^{\lambda n \frac{1-\rho}{r(1-\lambda)}} \left[ \frac{\ell(P_{\varepsilon, k}) }{\ell(Q)^\rho} \right]^{n/r} \mathrm{M}_rf(x)\\
        & =  [2^k\ell(Q)]^{-\frac{n}{r}\left(\varepsilon - \lambda \frac{1-\rho}{1-\lambda}\right)} \mathrm{M}_rf(x).
    \end{align*}
    Furthermore, by \cref{eq:f1-estimate} we obtain that
    \begin{align*}
        |T_{\sigma_k}f_{1, k}(y)| & \lesssim [2^{k\rho}\ell(P_{\varepsilon,k})]\mathrm{M}_rf(x)\\
        & \lesssim\left( 2^{k\rho}\ell(Q)^\rho [2^k\ell(Q)]^{-\varepsilon} 
        \right)^{-(N-n/r)} \mathrm{M}_rf(x) \\
        & \leq [2^k\ell(Q)]^{-(\rho-\varepsilon)(N-n/r)}\mathrm{M}_rf(x),        
    \end{align*}
    where $N>n/r$. Hence, we have a uniform bound for the second term. Now, let us consider $\mathcal{J}_3$. For this case we will use $P$ as the concentric dilate of $Q$ such that $\ell(P)=10\sqrt{n}\ell(Q)$, and use the same decomposition as in \cref{eq:decomposition}. Notice that since $\langle\xi\rangle \sim 2^k$ in the support of $\sigma_k$, then we have that
    \begin{equation*}
        2^{\frac{kn}{2}(1-\rho)}\sigma_k \in S_{\rho, \rho}^{-n(1-\rho)(1/r-1/2)} (\mathbb{T}^n \times \mathbb{Z}^n),
    \end{equation*}
    uniformly in $k$. Hence, \cite[Theorem~3.11]{Cardona:Martinez} implies the $L^r$ boundedness for every $T_{\sigma_k}$, and we obtain that
    \begin{align*}
        \left( \frac{1}{|Q|}\int_Q |T_{\sigma_k}f_0(y)|^r \diff y
        \right)^{1/r} & \leq \frac{1}{|Q|^{1/r}} \|T_{\sigma_k}f_0(y)\|_r \\
        & \lesssim 2^{-\frac{kn}{2}(1-\rho)} \frac{1}{\ell(Q)^{n/r}} \|f\chi_P\|_r \\
        & \lesssim2^{-\frac{kn}{2}(1-\rho)} \mathrm{M}_rf(x).
    \end{align*}
    On the other hand, we employ \cref{eq:f1-estimate} in order to get that
    \begin{equation*}
        |T_{\sigma_k}f_1(y)| \lesssim[2^{k\rho}\ell(Q)]^{-(N-n/r)} \mathrm{M}_rf(x),
    \end{equation*}
    where $N>n/r$. Hence we conclude that $\mathcal{J}_3 \lesssim\mathrm{M}_rf(x)$, and we complete the proof.
\end{proof}
In order to prove the $L^p(w)$ continuity of toroidal pseudo-differential operators, first we recall the following interpolation theorem from Bergh and L\"ofstrom.
\begin{theorem}
    \cite[Theorem~5.5.3]{bergh-lofstrom} Suppose that $1\leq p_0,p_1<\infty$. Then, for $0<\theta<1$ we have that 
    \begin{equation*}
        (L^{p_0}(w_0), L^{p_1}(w_1))_\theta = L^p(w),
    \end{equation*}
    where 
    \begin{equation*}
        \frac{1}{p} = \frac{1-\theta}{p_0} + \frac{\theta}{p_1} \quad \text{ and } \quad w = w_0^{\frac{p(1-\theta)}{p_0}} w_1^{\frac{p\theta}{p_1}}.
    \end{equation*}
    \label{theo:interpolation}
\end{theorem}
Now, we take advantage of \cref{eq:muckenhoupt} and \cref{theo:sharp-maximal} to obtain the following continuity result.
\begin{corollary}
    Let $0\leq \delta \leq \rho < 1$, let $0<\rho<1$, let $1<r\leq 2$, and let $r\leq p < \infty$. Suppose that $m\leq -n(1-\rho)/r$ and $\sigma\in S^m_{\rho,\delta}(\mathbb{T}^n\times\mathbb{Z}^n)$. If $w \in A_{p/r}$, then 
    \begin{equation*}
        \|T_\sigma f\|_{L^p(w)} \lesssim\|f\|_{L^p(w)},
    \end{equation*}
    for any $f\in C^\infty(\mathbb{T}^n)$.
\end{corollary}
\begin{proof}
    First, let us consider the case $p>r$. Hence, we obtain that
    \begin{equation*}
        \|T_\sigma f\|_{L^p(w)} \lesssim\| \mathcal{M}^\#_r(T_\sigma f)\|_{L^p(w)} \lesssim\|\mathrm{M}_rf\|_{L^p(w)} \lesssim\|f\|_{L^p(w)}.
    \end{equation*}
    Now, let us consider the case $p=r$. Then, for $w\in A_1$ we know \cite[Corollary~7.6]{duoandikoetxea} there exists $\varepsilon>0$ such that 
    \begin{equation*}
        w^{1+\varepsilon} \in A_1.
    \end{equation*}
    For any $r<q_0<\infty$, the embedding theory for $A_p$ weights implies that $w^{1+\varepsilon}\in A_{q_0/r}$, and by a similar argument as above we conclude that 
    \begin{equation*}
        \|T_\sigma f\|_{L^{q_0}(w^{1+\varepsilon})} \lesssim\|f\|_{L^{q_0}(w^{1+\varepsilon})}.
    \end{equation*}
    Moreover, from \cite[Theorem~3.12]{Cardona:Martinez} we have that 
    \begin{equation*}
        \|T_\sigma f\|_{L^{q_1}} \lesssim\|f\|_{L^{q_1}},
    \end{equation*}
    for any $1<q_1<r$, because $m \leq -n(1-\rho)(1/q_1 - 1/2)$. The desired result follows from the interpolation argument as in \cref{theo:interpolation}.
\end{proof}

Now, we employ \cref{theo:Lp-Lq} to obtain continuity properties of Sobolev spaces of periodic functions.
\begin{theorem}
    Let $0\leq\delta<1$, let $0<\rho\leq 1$, let $m\in\mathbb{R}$, and let $\sigma\in S^m_{\rho,\delta}(\mathbb{T}^n\times\mathbb{Z}^n)$. Then, $T_\sigma$ extends to a bounded operator from $W^{s}_p(\mathbb{T}^n)$ into $W^{s-\mu}_q (\mathbb{T}^n)$ where $1<p\leq q<\infty$, for any $s\in\mathbb{R}$, when 
    \begin{enumerate}
        \item $1<p\leq 2 \leq q$ and 
        \begin{equation*}
            \mu \geq m + n \left( \frac{1}{p} - \frac{1}{q} + \lambda
            \right),
        \end{equation*}
        \item if $2 \leq p \leq q$ and 
        \begin{equation*}
            \mu \geq m + n \left[ \frac{1}{p} - \frac{1}{q} + (1-\rho) \left( \frac{1}{2} - \frac{1}{p}
            \right)
            + \lambda
            \right],
        \end{equation*}
        \item if $p\leq q \leq 2$ and 
        \begin{equation*}
            \mu \geq m + n \left[ \frac{1}{p} - \frac{1}{q} + (1-\rho) \left( \frac{1}{q} - \frac{1}{2}
            \right)
            + \lambda
            \right],
        \end{equation*}
    \end{enumerate}
     where $\lambda := \max\{ 0, (\delta-\rho)/2 \}$. 
\end{theorem}
\begin{proof}
    Notice that $J^{s-\mu}T_\sigma J^{-s}$ has order $m-\mu$, which satisfies the requirements of \cref{theo:Lp-Lq}, implying its $L^p$-$L^q$-continuity. Hence, we obtain that
    \begin{equation*}
        \|T_\sigma f\|_{W^{s-\mu}_q} = \|T_\sigma J^{-s}J^sf\|_{W^{s-\mu}_q} 
        = \|J^{s-\mu}T_\sigma J^{-s}J^sf\|_{L^q}
         \lesssim \|J^sf\|_{L^p} = \|f\|_{W^s_p}. 
    \end{equation*}
    Thus, completing the proof.
\end{proof}
Now, we can use the fact that Besov spaces are the result of the real interpolation from Sobolev spaces to obtain the following result. 
\begin{corollary}
    Let $0\leq\delta<1$, let $0<\rho\leq 1$, let $m\in\mathbb{R}$, and let $\sigma\in S^m_{\rho,\delta}(\mathbb{T}^n\times\mathbb{Z}^n)$. Then, $T_\sigma$ extends to a bounded operator from $B^{s}_{p,r}(\mathbb{T}^n)$ into $B^{s-\mu}_{q, r}(\mathbb{T}^n)$ where $1<p\leq q<\infty$, for any $s\in\mathbb{R}$, and $1\leq r\leq \infty$, when,
    \begin{enumerate}
        \item $1<p\leq 2 \leq q$ and 
        \begin{equation*}
            \mu \geq m + n \left( \frac{1}{p} - \frac{1}{q} + \lambda
            \right),
        \end{equation*}
        \item if $2 \leq p \leq q$ and 
        \begin{equation*}
            \mu \geq m + n \left[ \frac{1}{p} - \frac{1}{q} + (1-\rho) \left( \frac{1}{2} - \frac{1}{p}
            \right)
            + \lambda
            \right],
        \end{equation*}
        \item if $p\leq q \leq 2$ and 
        \begin{equation*}
            \mu \geq m + n \left[ \frac{1}{p} - \frac{1}{q} + (1-\rho) \left( \frac{1}{q} - \frac{1}{2}
            \right)
            + \lambda
            \right],
        \end{equation*}
    \end{enumerate}
     where $\lambda := \max\{ 0, (\delta-\rho)/2 \}$. 
\end{corollary}

\section{Acknowledgments}
The first author thanks Bae Jun Park for discussions on the subject during his visit to the Ghent Analysis and PDE Center.

\bibliographystyle{amsplain}

\begin{thebibliography}{99}


\bibitem{agranovich}
Agranovich, M. S.: Spectral properties of elliptic pseudodifferential operators on a closed curve. \textit{Functional Analysis and Its Applications}, \textit{13}(4), 279–281. (1980) doi:10.1007/bf01078368 


\bibitem{alvarez-hounie}
Alvarez, J.,  Hounie, J.: Estimates for the kernel and continuity properties of pseudo-differential operators. \textit{Arkiv för Matematik}, \textit{28}(1–2), 1–22. (1990). doi:10.1007/bf02387364 


\bibitem{alvarez-milman}
Álvarez, J.,  Milman, M.: Vector valued inequalities for strongly singular Calderón-Zygmund operators. \textit{Revista Matemática Iberoamericana}, \textit{2}(4), 405–426. (1986). doi:10.4171/rmi/42 

\bibitem{bergh-lofstrom}
Bergh, J., Lofstrom, J. Interpolation spaces. \textit{Grundlehren der mathematischen Wissenschaften,} (1976).

\bibitem{calderon-vaillancourt-1}
Calder\'on, A. P.,  Vaillancourt, R.: On the boundedness of pseudo-differential operators. \textit{Journal of the Mathematical Society of Japan}, \textit{23}(2). (1971). doi:10.2969/jmsj/02320374 


\bibitem{calderon-vaillancourt-2}
Calder\'on, A. P.,  Vaillancourt, R.: A class of bounded pseudo-differential operators. \textit{Proceedings of the National Academy of Sciences}, \textit{69}(5), 1185–1187. (1972). doi:10.1073/pnas.69.5.1185 


\bibitem{calderon-zygmund}
Calder\'on, A. P.,  Zygmund, A.: On the existence of certain singular integrals. \textit{Acta Mathematica}, \textit{88}(0), 85–139. (1952). doi:10.1007/bf02392130 


\bibitem{cardona}
Cardona, D.: Weak type (1,1) bounds for a class of periodic pseudo-differential operators. \textit{Journal of Pseudo-Differential Operators and Applications}, \textit{5}(4), 507–515. (2014). doi:10.1007/s11868-014-0101-9 


\bibitem{Cardona2016} 
Cardona, D.: Hölder–Besov boundedness for periodic pseudo-differential operators. \textit{Journal of Pseudo-Differential Operators and Applications}, \textit{8}(1), 13–34. (2016). doi:10.1007/s11868-016-0174-8 


\bibitem{Cardona2018} 
Cardona, D.: On the boundedness of periodic pseudo-differential operators. \textit{Monatshefte für Mathematik}, \textit{185}(2), 189–206. (2018). doi:10.1007/s00605-017-1029-y 


\bibitem{cardona-delgado-kumar}
Cardona, D., Delgado, J., Kumar, V.,  Ruzhansky, M.: Lp-Lq-Boundedness of pseudo-differential operators on compact Lie groups, to appear in \textit{Osaka Journal of  Mathematics}, arXiv:2310.16247.

\bibitem{cardona:kumar:jfaa} 
Cardona, D.,  Kumar, V.: $L^p$-boundedness and $L^p$-nuclearity of multilinear pseudo-differential operators on ${\mathbb {Z}}^n$ and the torus ${\mathbb {T}}^n$. \textit{Journal of Fourier Analysis and Applications}, 25(6), 2973–3017. (2019). doi:10.1007/s00041-019-09689-7 

\bibitem{Cardona:Martinez} 
Cardona, D.  Martinez, M. A. Boundedness of pseudo-differential operators on the torus revisited, I. to appear in \textit{Journal of Mathematical Analysis and Applications},  arXiv:2502.20575.

\bibitem{Cardona:Martinez-2}
Cardona, D., Martinez, M. A.: Boundedness of pseudo-differential operators on the torus revisited, II.  arXiv:2505.01573

\bibitem{CardonaCRAS} 
Cardona, D., Messiouene, R.,  Senoussaoui, A.: Periodic Fourier integral operators in Lp-spaces. \textit{Comptes Rendus. Mathématique}, \textit{359}(5), 547–553. (2021b). doi:10.5802/crmath.194 


\bibitem{cardona-ruzhansky}
Cardona, D.,  Ruzhansky, M.: Oscillating Singular integral operators on compact lie groups revisited. \textit{Mathematische Zeitschrift}, \textit{303}(2). (2022). doi:10.1007/s00209-022-03175-5 


\bibitem{cardona-ruzhansky-subelliptic}
Cardona, D.,  Ruzhansky, M.: Subelliptic pseudo-differential operators and Fourier integral operators on compact Lie groups. To appear in \textit{MSJ Memoirs, Mathematical Society of Japan}, Tokyo, 180 Pages. arXiv:2008.09651.

\bibitem{chanillo-torchinsky}
Chanillo, S., Torchinsky, A.: Sharp function and weighted Lp estimates for a class of pseudo-
differential operators, \textit{Arkiv Mathematik} 24 (1985), 1-25

\bibitem{cobos-ruzhansky}
G\'omez Cobos, S.  Ruzhansky, M.:
$L^p$-bounds in Safarov pseudo-differential calculus on manifolds with bounded geometry. To appear in \textit{Journal of London Mathematical Society},  arXiv:2403.13920. 

\bibitem{david-journe}
David, G.,  Journ\'e, J. L.: A boundedness criterion for generalized Calder\'on-Zygmund operators. \textit{The Annals of Mathematics}, \textit{120}(2), 371. (1984). doi:10.2307/2006946 


\bibitem{delgado}
Delgado, J.: Lp-bounds for pseudo-differential operators on the torus. \textit{Pseudo-Differential Operators, Generalized Functions and Asymptotics}, 103–116. (2013). doi:10.1007/978-3-0348-0585-8\_6 


\bibitem{delgado-ruzhansky}
Delgado, J.,  Ruzhansky, M.: Lp-bounds for pseudo-differential operators on compact lie groups. \textit{Journal of the Institute of Mathematics of Jussieu}, \textit{18}(3), 531–559. (2017). doi:10.1017/s1474748017000123 


\bibitem{duoandikoetxea}
Duoandikoetxea, J.: Fourier Analysis. \textit{American Mathematical Society, Providence}. (2001).



\bibitem{fefferman-BMO}
Fefferman, C.: Characterizations of bounded mean oscillation. \textit{Bulletin of the American Mathematical Society}, \textit{77}(4), 587–588. (1971). doi:10.1090/s0002-9904-1971-12763-5 


\bibitem{fefferman-Lp}
Fefferman, C.: Lp bounds for pseudo-differential operators. \textit{Israel Journal of Mathematics}, \textit{14}(4), 413–417. (1973). doi:10.1007/bf02764718 


\bibitem{fefferman-stein}
Fefferman, C.,  Stein, E. M.: Hp spaces of several variables. \textit{Acta Mathematica}, \textit{129}(0), 137–193. (1972). doi:10.1007/bf02392215 


\bibitem{folland-stein}
Folland, G. B.,  Stein, E. M.: Hardy spaces on homogeneous groups. \textit{Princeton University Press, Princeton, NJ}. (1982). doi:10.1515/9780691222455 

\bibitem{garcia-defrancia}
Garcia-Cuerva, J., Rubio de Francia, J. L.: Weighted norm inequalities and related topics. \textit{North-Holland.} (1985)

\bibitem{hormander}
Hörmander, L.: The Analysis of the Linear Partial Differential Operators, vol. III.  \textit{Springer, Berlin}


\bibitem{hounie}
Hounie, J.: On the $L^2$ continuity of pseudo-differential operators. \textit{Communications in Partial Differential Equations, 11}(7), 765–778. (1986). doi:10.1080/03605308608820444 



\bibitem{kohn-nirenberg}
Kohn, J., Nirenberg, L.: {An algebra of pseudo-differential operators}. \textit{Communications on Pure and Applied Mathematics},
XVIII, 269–305. (1965).

\bibitem{mclane}
McLean, W.M.: Local and global description of periodic pseudo-differential operators. \textit{Mathematische Nachrichten}
150, 151–161. (1991).

\bibitem{miyachi-yabuta}
Miyachi, A., Yabuta, K.: Sharp function estimates for pseudo-differential operators of class $S^m_{\rho,\delta}$,
\textit{Bull. Fac. Sci. Ibaraki Univ. Ser.} A 19 (1987), 15-30.

\bibitem{molahajloo-wong}
Molahajloo, S.,  Wong, M.W.: Pseudo-differential operators on $\mathbb{S}^1$. \textit{New Developments in Pseudo-differential Operators}, editors: Rodino, L., Wong, M.W., 297–306. (2008).

\bibitem{muckenhoupt}
Muckenhoupt, B.: Weighted norm inequalities for the Hardy maximal function. \textit{Transactions of the American Mathematical Society} 165, 207-226. (1972)

\bibitem{nagase}
Nagase, M.: On some classes of L p-bounded pseudodifferential operators. \textit{Osaka Journal of Mathematics 23}(2),
425-440. (1986)

\bibitem{park-tomita} 
Park, B. J., Tomota, N.: Sharp maximal function estimates for linear and multilinear pseudo-differential operators. \textit{Journal of Functional Analysis 287(12)}. (2024) 

\bibitem{defrancia}
Rubio de Francia, J. L., Ruiz, F. J.,  Torrea, J. L. {Calderón-Zygmund theory for operator-valued kernels.} \textit{Advances in Mathematics 62}(1), 7–48. (1986). doi:10.1016/0001-8708(86)90086-1 


\bibitem{ruzhansky-turunen2}
Ruzhansky, M.,  Turunen, M.: On the Fourier analysis of operators on the torus.\textit{ }\textit{Modern Trends in Pseudo-Differential Operators}, editors: J. Toft, M.W. Wong,
and H. Zhu, Operator Theory: Advances and Applications 172, Birkh\"auser,
87–105. (2007).

\bibitem{ruzhansky-turunen}
Ruzhansky, M.,  Turunen, V.: Pseudo-differential operators and symmetries: Background analysis and advanced topics. \textit{Basel: Birkh\"auser. } (2010).

\bibitem{ruzhansky-turunen-quant}
Ruzhansky, M.,  Turunen, V.: Quantization of Pseudo-Differential Operators on the torus. \textit{Journal of Fourier
Analysis and Applications 16}, 943–982. Birkhäuser Verlag, Basel. (2010).

\bibitem{turunen-vainikko}
Turunen, V.,  Vainikko, G.: On symbol analysis of periodic pseudodifferential operators. \textit{Zeitschrift für Analysis und ihre Anwendungen} 17, 9–22. (1998).

\bibitem{wang}
Wang, L.: Pseudo-differential operators with rough coefficients. \textit{Ph. D. Thesis, McMaster University}. (1997).

\bibitem{wong}
Wong M. W.: Discrete Fourier Analysis. \textit{Birkhäuser}. (2011).


\end{thebibliography}

\end{document}